\documentclass[11pt]{amsart}
\usepackage{amsmath,amssymb}
\numberwithin{equation}{section}
 \usepackage{xcolor}

\newtheorem{theorem}{Theorem}[section]
\newtheorem{thm}[theorem]{Theorem}

\newtheorem{lem}[theorem]{Lemma}
\newtheorem{prop}[theorem]{Proposition}

\theoremstyle{definition}

\newtheorem{defi}[theorem]{Definition}

\newtheorem{exa}[theorem]{Example}

 \newtheorem*{ackn}{Acknowledgements}
 
  \newtheorem*{thmintro}{Theorem} 
   \newtheorem*{corintro}{Corollary}

 \newtheorem*{introdef}{Definition}

 \theoremstyle{plain}
\newtheorem*{namedthm}{\namedthmname}
\newcounter{namedthm}

\makeatletter

 \newcommand{\R}{\mathbb R}
 
 \newcommand{\C}{\mathbb C}
  \newcommand{\PP}{\mathbb P}
 \newcommand{\N}{\mathbb N}

 \newcommand{\e}{\varepsilon}

 \newcommand{\f}{\varphi}
 
 \newcommand{\p}{\psi}

 \newcommand \psh {plurisubharmonic }
  \newcommand \MA {Monge-Amp\`ere }

 \newcommand \PSH {{\rm PSH}}

 \usepackage{hyperref}
\hypersetup{
    unicode=false,        
    pdftoolbar=true,      
    pdfmenubar=true,       
    pdffitwindow=false,     
    pdfstartview={FitH},    
    pdftitle={Uniform estimates for CMAE},    
    pdfauthor={Guedj, Lu},     
    colorlinks=true,       
   linkcolor=blue,          
    citecolor=blue,        
    filecolor=black,      
    urlcolor=blue}



\begin{document}


\title[Quasi-monotone convergence]{Quasi-monotone convergence of plurisubharmonic functions}

\author{Vincent Guedj \& Antonio Trusiani}

\address{Institut de Math\'ematiques de Toulouse   \\ Universit\'e de Toulouse \\
118 route de Narbonne \\
31400 Toulouse, France\\}

\email{\href{mailto:vincent.guedj@math.univ-toulouse.fr}{vincent.guedj@math.univ-toulouse.fr}}
\urladdr{\href{https://www.math.univ-toulouse.fr/~guedj}{https://www.math.univ-toulouse.fr/~guedj/}}


\email{\href{mailto:antonio.trusiani@math.univ-toulouse.fr}{antonio.trusiani@math.univ-toulouse.fr}}
\urladdr{\href{https://sites.google.com/view/antonio-trusiani/home}{https://sites.google.com/view/antonio-trusiani/home}}
\date{\today}

 \begin{abstract}
 The complex Monge-Amp\`ere operator has been defined for locally bounded plurisubharmonic functions
 by Bedford-Taylor in the 80's. This definition has been extended to compact complex manifolds,
 and to various classes of  mildly unbounded quasi-plurisubharmonic functions by various authors.
 As this operator is not continuous for the $L^1$-topology, several
stronger topologies have been introduced over the last decades to remedy this, while maintaining
efficient compactness criteria. The purpose of this note is to show that these stronger
topologies are essentially equivalent to the  natural quasi-monotone topology that we introduce 
and study here.
 \end{abstract}

 \maketitle


\section*{Introduction}

 In connection with the spectacular developments of K\"ahler geometry
 in the last decade (see  \cite{CDS15,Don18,BBEGZ,BBJ21,LTW20,Li22,LWZ22}),
  several finite energy spaces of 
 quasi-\psh functions have been studied, each of which endowed with
 a strong topology, that ensures completeness of the space,
 good compactness criteria, and continuity of the complex Monge-Amp\`ere
 operator (the latter being discontinuous for the weaker
 $L^1$-topology).
 In this note we introduce yet another strong convergence,
 the {\it quasi-monotone convergence}. 
 
 \smallskip
 
 Let $(X,\omega)$ be a compact K\"ahler manifold.
 A function $\f:X \rightarrow \R \cup \{-\infty\}$
 is quasi-\psh if it is locally  the sum of smooth and a \psh function.
 The function $\f$ is called $\omega$-psh if 
 $\omega+dd^c \f \geq 0$ in the sense of currents.
 We let $PSH(X,\omega)$ denote the set of all $\omega$-psh functions;
 it can be endowed with the $L^1$-topology,
 which is equivalent to the weak topology of distributions.
  If $\f_j \in PSH(X,\omega)$ converges in $L^1$ to $\f \in PSH(X,\omega)$,
 then $\f_j^+:=\left( \sup_{\ell \geq j} \f_{\ell} \right)^*$ is a decreasing
 sequence of $\omega$-psh functions that pointwise decreases to $\f$
 (see \cite[Proposition 8.4]{GZbook}).
 The quasi-monotone convergence requires a dual property:
 
 \begin{introdef}
 We say that $\f_j$ converges 
 quasi-monotonically if
 $
 \f_j^-:=P_{\omega}\left(\inf_{\ell \geq j} \f_{\ell} \right)
 $
is a sequence of  $\omega$-\psh functions that increases to $\f$.
 \end{introdef}
 
 Here $P_{\omega}(h)$ denotes the $\omega$-psh envelope of the function $h$: this is 
 the largest $\omega$-psh function that lies below $h$. The family of such functions
 is compact, but it may happen that it is empty in which case 
 $P_{\omega}(h) \equiv -\infty$ (see Example \ref{exa:infty}).
 
  \smallskip
 
 The complex Monge-Amp\`ere operator ${\rm MA}$ is well defined for bounded 
 $\omega$-psh functions, as follows from Bedford-Taylor theory
 \cite{BT82,GZbook}. At the heart of the theory lies the continuity
 property of ${\rm MA}$ along monotone sequences.
 Our convergence notion naturally extends this property, allowing for
 sequences $\f_j$ that are bounded from above and below $u_j \leq \f_j \leq v_j$
 by a sequence of $\omega$-psh functions $u_j$ (resp. $v_j$) which increases (resp. decreases)
to $\f$  [see Lemma \ref{lem:def}].

 The property that we require is  somewhat dual to the property for the upper-envelope,
 which always holds. This is well illustrated by the case when the $\f_j$'s
are solutions of a complex Monge-Amp\`ere equation (or of a complex Monge-Amp\`ere flow), 
where this operation (envelope of infima) leaves the space of super-solutions invariant
(see \cite[Theorem C]{GLZ19} and Examples \ref{exa:solutions} and \ref{exa:fkr}).

\smallskip

Our main result compares all these strong topologies : 

\begin{thmintro}
Let $(X,\omega)$ be a compact K\"ahler manifold, and
$(\f_j) \in PSH(X,\omega)^{\N}$ a  sequence which converges in $L^1(X)$ to
some $\f \in PSH(X,\omega)$. 
Fix $\p \in {\mathcal E}^1(X,\omega)$.
\begin{enumerate}
\item If $\f_j$ converges quasi-monotonically then $\f_j$ converges in capacity.

\item If $\f_j$ converges in capacity and $\f_j \geq \p$ then 
$\f_j$ converges  in $({\mathcal E}^{1}(X,\omega),d_{1})$.

\item If  $\f_j$ converges  in $({\mathcal E}^{1}(X,\omega),d_{1})$, then
$\f_j$ converges to $\f$ in capacity. Moreover
up to extracting and relabelling, the convergence is quasi-monotone
and there exists $\tilde{\p} \in {\mathcal E}^1(X,\omega)$ such that $\f_j \geq \tilde{\p}$.
\end{enumerate}
\end{thmintro}

A sequence $(\f_j)$ converges in capacity to $\f$ if for all $\delta>0$,
$$
{\rm Cap}_{\omega}\left( \left\{ z \in X, \; | \f_j(x)-\f(x)| \geq \delta \right\} \right)
\longrightarrow 0,
$$
as $j \rightarrow +\infty$,
where ${\rm Cap}_{\omega}$ denotes  the Monge-Amp\`ere capacity introduced by
Kolodziej \cite{Kol03} and further studied in
\cite{GZ05,H08,Xing09,DH12}.

The set $({\mathcal E}^{1}(X,\omega),d_{1})$ denotes the finite energy space introduced
in \cite{GZ07}, endowed with the metrizable strong topology considered in \cite{BBEGZ}
and further studied in \cite{Darv15}.
One can equally well consider other finite energy classes
$({\mathcal E}_{\chi}(X,\omega),d_{\chi})$,
endowed with the induced Mabuchi topology, and show an appropriate version of the
above result.
We refer the reader to Definitions  \ref{def:cap} and \ref{def:echi},
and to Theorem \ref{thm:energy} for 
the precise statements.

All the   previous notions of strong convergence
therefore coincide --up to extracting and relabelling-- when the sequence is uniformly bounded:

\begin{corintro}
Let $(X,\omega)$ be a compact K\"ahler manifold, and
$(\f_j) \in PSH(X,\omega)^{\N}$ a uniformly bounded sequence 
which converges in $L^1(X)$ 
to some $\f \in PSH(X,\omega)$. Up to extracting and relabelling, 
the following properties are equivalent
\begin{enumerate}
\item $\f_j$ converges quasi-monotonically to $\f$;

\item  $\f_j$ converges   to $\f$ in capacity;

\item  $\f_j$ converges   to $\f$ in $({\mathcal E}^{1}(X,\omega),d_{1})$.
\end{enumerate}
\end{corintro}

Example \ref{exa:infty} shows that a sequence can converge in capacity
but not quasi-monotonically. The sequence $\f_j=\f/j$ converges to zero
both in capacity and quasi-monotonically, but not in energy if $\f$ does not belong to 
${\mathcal E}^1(X,\omega)$.
Example \ref{exa:extraction} shows that   extracting is necessary to ensure
the quasimonotone convergence. Finally Example \ref{exa:pascap} provides an example 
where these strong convergences do not hold, while Example \ref{exa:energy}
compares the various types of convergence in energy.


\subsubsection*{Contents}

 We recall basic facts 
 in Section \ref{sec:cap}. The quasi-monotone convergence is introduced
 in Section \ref{sec:qm}, where we prove our main Theorem. We provide several explicit
 examples in Section \ref{sec:exa} and explain how our observations extend to big cohomology
 classes with prescribed singularities, as well as to the local setting.

   \begin{ackn}
   We thank E.DiNezza, H.C.Lu and A.Zeriahi for useful comments. 
The authors are partially supported by the research project Hermetic (ANR-11-LABX-0040),
the ANR projects Paraplui. The second author is supported by a grant from the
Knut and Alice Wallenberg foundation.
\end{ackn}

    \section{Capacity and energies} \label{sec:cap}
    
 In the whole article we let $X$ denote a compact K\"ahler manifold of complex dimension
  $n \geq 1$, and we fix  $\omega$ denote a K\"ahler form on $X$.
    
    \subsection{Convergence in capacity}
    
    \subsubsection{Quasi-plurisubharmonic functions}
    
    A function is quasi-plurisub\-harmonic if it is locally given as the sum of  a 
    smooth and a psh function.    Quasi-psh functions
$\f:X \rightarrow \R \cup \{-\infty\}$ satisfying
$\omega+dd^c \f \geq 0$
in the weak sense of currents are called $\omega$-plurisubharmonic ($\omega$-psh for short).
Here $d=\partial+\overline{\partial}$ and $d^c=\frac{1}{2i\pi}(\partial-\overline{\partial})$.

A ${\mathcal C}^2$-smooth function $u$ has bounded Hessian, hence $\e u$ is
$\omega$-psh if $0<\e$ is small enough. 
Note that constants functions are also $\omega$-psh functions.

\begin{defi}
We let $\PSH(X,\omega)$ denote the set of all $\omega$-plurisubharmonic functions which are not identically $-\infty$.  
\end{defi}

The set $\PSH(X,\omega)$ is a closed subset of $L^1(X)$, 
for the $L^1$-topology.
Subsets of  $\omega$-psh functions enjoy strong compactness and integrability properties,
we mention notably the following: for any fixed $r \geq 1$, 
\begin{itemize}
\item $\PSH(X,\omega) \subset L^r(X)$; the induced $L^r$-topologies are all equivalent;
 \item $\PSH(X,\omega) \subset W^{1,q}(X)=\{ u \in L^1(X), \; \nabla u \in L^q(X)\}$ for all $q<2$;
 the induced $W^{1,q}$-topology is again equivalent to the $L^1$-topology;
\item $\PSH_A(X,\omega):=\{ u \in\PSH(X,\omega), \, -A \leq \sup_X u \leq 0 \}$ is compact.
\end{itemize}
We refer the reader to \cite{GZbook} for further basic properties of $\omega$-psh functions.

\smallskip

      Let $\f$ be a \psh function in $\C^n$.
 If $\f$ is smooth, its complex Monge-Amp\`ere measure is
 defined by
 $$
 MA(\f)=(dd^c \f)^n=c_n \det \left( \frac{\partial^2 \f}{\partial z_j \partial \overline{z_k}}
 \right) dV_{eucl},
 $$
 where 
 $dV_{eucl}$ is the euclidean volume form, and $c_n>0$ is a normalizing constant.
 
 When $\f$ is less regular, one can approximate it from above by smooth \psh functions
 $\f_{\e_j}=\f \star \chi_{\e_j}$ obtained by convolution with a standard family of mollifiers.
 If the measures $MA(\f_{\e_j})$ converges to a limit $\mu_{\f}$, one   sets $MA(\f):=\mu_{\f}$. 
 This definition has been shown to be consistent by Bedford and Taylor \cite{BT82}
 when $\f$ is {\it locally bounded}: in this case they have shown that $\mu_{\f}$ is the limit
 of {\it any} decreasing sequence of \psh approximants.
 
 Bedford-Taylor's theory has been adapted 
 to  the compact setting (see \cite{GZbook,Din16}), 
 and the definition of $MA$ has been extended to
 mildly unbounded quasi-\psh functions 
 (see \cite{Ceg04, Blo06,CGZ08,GZbook}). 
  In the compact setting,
 one needs to replace \psh functions by quasi-\psh ones.
 One can approximate $\f \in PSH(X,\omega)$ by a decreasing sequence of smooth 
 $\omega$-psh functions $\f_j$ \cite{Dem92} and set $V=\int_X \omega^n$ and 
 $$
 MA(\f):=\lim_{j \rightarrow +\infty} V^{-1}(\omega+dd^c \f_j)^n,
 $$
 whenever the limit is well-defined and independent of the approximants.
 
The complex Monge-Amp\`ere measure 
$
MA(\f)
$
 is in particular well-defined for any
$\omega$-psh function $\f$ which is {\it bounded}.
It is also well-defined for unbounded $\omega$-psh functions
that have {\it finite energy} (see Section \ref{sec:energy}).

     \subsubsection{The Monge-Amp\`ere capacity}

 \begin{defi}  \label{def:cap}
 Given $K \subset X$ a compact set, its
Monge-Amp\`ere capacity is  
  $$
  {\rm Cap}_{\omega}(K):=\sup \left\{ \int_K (\omega+dd^c u)^n, \; u \in PSH(X,\omega)
  \text{ with } 0 \leq u \leq 1 \right\}.
  $$
  \end{defi}
  
  This notion has been introduced in \cite{Kol03} and further studied in \cite{GZ05}. 
  We refer the reader to    \cite[Chapter 9]{GZbook} for its basic properties.

  \begin{defi} 
  A sequence $(\f_j) \in PSH(X,\omega)^{\N}$ converges in capacity
  to $\f \in PSH(X,\omega)$ if for all $\delta >0$,
  $$
  {\rm Cap}_{\omega}\left( \{x \in X, \; |\f_j(x)-\f(x)| \geq \delta \} \right)
  \stackrel{j \rightarrow +\infty}{\longrightarrow} 0.
  $$
  \end{defi}
  
  It is known that convergence in capacity implies convergence in $L^1$
  (see \cite[Lemma 4.24]{GZbook}). One can moreover reduce to   uniformly bounded
  sequences:
  
  \begin{prop}
   A sequence $(\f_j) \in PSH(X,\omega)^{\N}$ converges in capacity
  to $\f \in PSH(X,\omega)$ if and only if for all $C>0$, 
  the sequence $\max(\f_j,-C)$ converges in capacity to $\max(\f,-C)$.
   \end{prop}
  
  \begin{proof}
  We can assume without loss of generality that $\f,\f_j \leq 0$.
  It follows from the Chern-Levine-Nirenberg inequalities 
   \cite[Corollary 9.5]{GZbook}
  that  
  $$
   {\rm Cap}_{\omega}\left( \{x \in X, \; \f_j(x)<-C \} \right) \leq 
   \frac{||\f_j||_{L^1}+nV}{C}
  $$
  and similarly ${\rm Cap}_{\omega}\left( \{\f<-C \} \right) \leq 
   \frac{||\f||_{L^1}+nV}{C}$. Thus 
   $
   \{x \in X, \; |\f_j(x)-\f(x)| \geq \delta \} 
   $
   and 
   $ \{x \in X, \; |\max(\f_j(x),-C)-\max(\f(x),-C)| \geq \delta \} $
   differ by a set whose capacity is uniformly small in $j$, as $C$ goes to $+\infty$.
 The conclusion follows.
  \end{proof}

  For uniformly bounded sequences, the convergence in capacity implies the convergence
  of Monge-Amp\`ere measures \cite[Theorem 1]{Xing96}.
  More generally we have the following consequence of \cite[Theorem 4.26]{GZbook}:
  
  \begin{prop} \label{pro:cvcap}
 Let $\f_j^{\ell}$ be a uniformly bounded sequence of $\omega$-psh functions
 which converge in capacity to $\f^{\ell} \in PSH(X,\omega)$, $0 \leq \ell \leq n$.
 For all continuous weight $\chi: \R \rightarrow \R$, the
 weighted measure
 $\chi(\f_j^0) (\omega+dd^c \f_j^{1}) \wedge \cdots \wedge (\omega+dd^c \f_j^n)$ 
 weakly converge to the weighted   measure
 $\chi(\f^0) (\omega+dd^c \f^{1}) \wedge \cdots \wedge (\omega+dd^c \f^n)$.
  \end{prop}

 We shall need the following estimate, which is an adaptation to the 
 compact setting of a local observation of 
 Blocki \cite{Blo93}.
 
 \begin{prop} \label{pro:blocki}
 Let $u,v,w$ be $\omega$-psh functions such that $-1 \leq u \leq 0$
 and $v \leq w$. Then
 $$
 \int_X (w-v)^{n+1} (\omega+dd^c u)^n \leq (n+1) ! 
 \sum_{j=0}^n \int_X (w-v)^{n+1-j} \, (\omega+dd^c v)^j \wedge \omega^{n-j}.
 $$
 \end{prop}

 \begin{proof}
 Set $\omega_u=\omega+dd^c u$ and $\omega_v=\omega+dd^c v$.
 Let $S$ be a closed current of bidegree $(1,1)$.
 By induction it suffices to establish the following inequality,
  $$
 \int_X (w-v)^{p+1} \omega_u \wedge S \leq  
   \int_X (w-v)^{p+1} \, \omega \wedge S
  +(p+1)  \int_X (w-v)^{p} \, \omega_v \wedge S.
 $$
 and apply it to $S=\omega_u^a \wedge \omega_v^b \wedge \omega^c$.
  This follows from Stokes theorem, observing that 
  $
  -dd^c (w-v)^{p+1} \leq (p+1) (w-v)^p \omega_v,
  $
  hence 
  
  \smallskip
  
  $
  \int_X (w-v)^{p+1} dd^c u \wedge S
  =\int_X u \, dd^c (w-v)^{p+1}  \wedge S \leq 
  (p+1) \int (w-v)^p \omega_{v} \wedge S.
  $
 \end{proof}

   \subsection{Finite energy topologies} \label{sec:energy}
  
  \subsubsection{Finite energy classes}
  
  Given   $\f \in \PSH(X,\omega)$, we consider  
$$
\f_j:=\max(\f, -j) \in \PSH(X,\omega) \cap L^{\infty}(X).
$$
The measures $MA( \f_j)$
are well defined probability measures and the sequence
$
\mu_j:={\bf 1}_{\{ \f>-j\}} MA(\f_j)
$
is increasing \cite[p.445]{GZ07}, with total mass bounded from above by $1$.
We consider
$$
\mu_{\f}:=\lim_{j \rightarrow +\infty} \mu_j,
$$
which is a positive Borel measure on $X$, with total mass $\leq 1$.

\begin{defi}
We set 
$
{\mathcal E}(X,\omega):=\left\{ \f \in \PSH(X,\omega) \; | \; \mu_{\f}(X)=1 \right\}.
$

For $\f \in {\mathcal E}(X,\omega)$, we set $MA(\f):=\mu_{\f}$.  
\end{defi}

It is proved in \cite{GZ07} that the  Monge-Amp\`ere operator  $MA$ is well defined on the class
${\mathcal E}(X,\omega)$.
One has a stratification
$$
{\mathcal E}(X,\omega)=\bigcup_{\chi \in {\mathcal W}} {\mathcal E}_{\chi}(X,\omega),
$$
where ${\mathcal W}$ denotes the set of all functions $\chi:\R \rightarrow \R$ such that $\chi$ is increasing and $\chi(-\infty)=-\infty$, and the finite energy class
${\mathcal E}_{\chi}(X,\omega)$ is defined as follows:

\begin{defi} \label{def:echi}
We set 
${\mathcal E}_{\chi}(X,\omega):=\left\{ \f \in {\mathcal E}(X,\omega) \; | \; \chi(-|\f|) \in L^1({\rm MA}(\f)) \right\}$.

When $\chi(t)=-(-t)^p$, $p>0$, we set ${\mathcal E}^p(X,\omega)={\mathcal E}_\chi(X,\omega)$. 
\end{defi}

  The set ${\mathcal E}^1(X,\omega)$  can be characterized as the set of 
$\f \in {\mathcal E}(X,\omega)$ such that 
$$
E(\f):=\frac{1}{n+1} \sum_{j=0}^n \int_X \f (\omega+dd^c \f)^j \wedge \omega^{n-j} >-\infty.
$$
Observe that 
$\bigcap_{\chi \in {\mathcal W}} {\mathcal E}_{\chi}(X,\omega)=PSH(X,\omega) \cap L^{\infty}(X)$,
so that finite energy classes interpolate between ${\mathcal E}(X,\omega)$
and bounded $\omega$-psh functions.

\smallskip


It follows from \cite[Theorem C]{GZ07}  that a  probability measure $\mu$ is the Monge-Amp\`ere measure of a potential in $\mathcal{E}^p(X,\omega)$ if and only if $\mathcal{E}^p(X,\omega) \subset L^p(X,\mu)$,
while \cite{GZ07,Din09} ensures that
 $\mu$ does not charge pluripolar sets if and only if
there exists a unique $\f \in {\mathcal E}(X,\omega)$ such that
$\mu=(\omega+dd^c \f)^n$ with $\sup_X \f=0$.

\begin{exa}
Every bounded $\omega$-psh function  belongs to ${\mathcal E}(X,\omega)$.
The class ${\mathcal E}(X,\omega)$ also contains many $\omega$-psh functions which are unbounded:
\begin{itemize}
\item when $X$ is a compact Riemann surface, ${\mathcal E}(X,\omega)$ is precisely the set of $\omega$-sh functions
whose Laplacian does not charge polar sets.
\item if $\f \in \PSH(X,\omega)$  satisfies $\f \leq -1$, then $\f_{\e}=-(-\f)^\e $ belongs to
${\mathcal E}(X,\omega)$ whenever $0 \leq \e <1$,
and $\f_{\e}$ belongs to ${\mathcal E}^p(X,\omega)$ if $\e<1/(n+p)$.
\item the functions in ${\mathcal E}(X,\omega)$
have relatively mild singularities; in particular they have zero Lelong number at every point.  
\end{itemize}
\end{exa}

\subsubsection{Mabuchi geometry}

The class ${\mathcal E}^1(X,\omega)$ has played a key role in recent applications of pluripotential
theory to K\"ahler geometry (see e.g. \cite{BBJ21}). Set
$$
I(\f,\p)=\int_X (\f-\p) [MA(\p)-MA(\f)] \geq 0.
$$
This quantity is well defined for $\f,\p \in {\mathcal E}^1(X,\omega)$ and satisfies a quasi-triangle
inequality \cite[Theorem 1.8]{BBEGZ}, hence induces a distance $d_I$.

\begin{defi}
The strong topology on ${\mathcal E}^1(X,\omega)$ is the  one  induced by $d_I$.
\end{defi}

This notion has been introduced in \cite[Section 5.3]{BBGZ13}, it implies convergence
in capacity \cite[Theorem 5.7]{BBGZ13}.
A sequence $(\f_j) \in {\mathcal E}^1(X,\omega)^{\N}$
strongly converges to $\f$ iff it converges in $L^1$ and $E(\f_j)$ converges to $E(\f)$.
Moreover the metric space $({\mathcal E}^1(X,\omega),d_I)$ is complete
\cite[Propositions 2.3 and 2.4]{BBEGZ}.

\smallskip

Let ${\mathcal H}=\{ \f \in {\mathcal C}^{\infty}(X,\R), \; \omega+dd^c \f >0\}$ denote the
set of smooth and strictly $\omega$-psh functions ({\it K\"ahler potentials}).
This set can be thought of as 
 an infinite dimensional Riemannian manifold, whose
tangent space at $\f \in {\mathcal H}$ can be identified with ${\mathcal C}^{\infty}(X,\R)$.
Following earlier work of Mabuchi, Darvas \cite{Darv15} has considered the following Finsler structure:
for $f \in T_{\f}{\mathcal H}$ he sets
$
|f|_{\f}:=\int_X |f| MA(\f).
$
If $\gamma:[0,1] \rightarrow {\mathcal H}$ is a Lipschitz path, one then defines
$$
\ell(\gamma)=\int_0^1 \int_X |\gamma'(t)| MA(\gamma(t)) dt,
$$
and given $\f,\p \in {\mathcal H}$, one   considers
$$
d_1(\f,\p):=\inf \left\{ \ell(\gamma), \; \gamma:[0,1] \rightarrow {\mathcal H}
\text{ with } \gamma(0)=\f, \; \gamma(1)=\p \right\},
$$
where the infimum runs over all Lipschitz paths joining $\f$ to $\p$.
The following summarizes some results of \cite{Darv15} that we shall need:

\begin{thm} \label{thm:darvE1}
Fix $\f,\p \in {\mathcal E}^1(X,\omega)$. The following properties hold:
\begin{itemize}
\item $d_1$ is a distance on ${\mathcal H}$ which is uniformly equivalent to $d_I$;
\item $d_1$ uniquely extends to ${\mathcal E}^1(X,\omega)$, 
$({\mathcal E}^1(X,\omega),d_1)$ is a geodesic metric space;
\item $P_{\omega}(\min(\f,\p)) \in {\mathcal E}^1(X,\omega)$ with
$d_1(\f, P_{\omega}(\min(\f,\p))) \leq d_1(\f,\p)$;
\item if $\p \leq \f$, 
then $d_1(\f, \p)$ is comparable to  $\int_X (\f-\p)MA(\p)$.
\end{itemize}
\end{thm}

 We refer   to \cite[Theorem 2, Theorem 3, Corollary 4.14]{Darv15} 
for more details.

\medskip

Analogous strong topologies have been defined on the other energy classes,
notably the classes ${\mathcal E}^p(X,\omega)$
(see \cite{Darv15,GLZ19,Darv21}).
If $\chi \in {\mathcal W}$ is {\it convex} with polynomial growth at infinity, the class $\mathcal{E}_{\chi}(X,\omega)$ can be equipped with a Finsler metric $d_{\chi}$ making it a complete geodesic metric space \cite{Darv15}.  
The Mabuchi distance $d_{\chi}$ is again comparable to a  pluripotential quasi-distance,
\begin{equation*} 
C^{-1}d_{\chi}(u,v) \leq  I_{\chi}(u,v):=\int_X |\chi(u-v)| (MA(u)+MA(v)) \leq C d_{\chi}(u,v), 
\end{equation*}
for a constant $C=C(\chi)>0$ and 
$\ u,v\in \mathcal{E}_{\chi}(X,\omega)$. 

{\it Concave} weights $\chi$ correspond to {\it low energy classes}. These are the weights  
one needs to  consider for the stratification of the class ${\mathcal E}(X,\omega)$.
One can still consider $I_{\chi}$ and $d_{\chi}$, but the distance $d_{\chi}$ is no longer induced
by a Finsler metric as emphasized in \cite[p2]{Darv21}. 
The following summarizes the results obtained by Darvas in \cite{Darv21} that we shall need:

\begin{thm} \label{thm:darvEchi}
Fix $\f,\p \in {\mathcal E}_{\chi}(X,\omega)$. The following properties hold:
\begin{itemize}
\item $d_{\chi}$ is a distance on ${\mathcal H}$ which is uniformly equivalent to $I_{\chi}$;
\item $d_{\chi}$ uniquely extends to ${\mathcal E}_{\chi}(X,\omega)$, 
 $({\mathcal E}_{\chi}(X,\omega),d_{\chi})$ is a complete metric space;
\item $P_{\omega}(\min(\f,\p)) \in {\mathcal E}_{\chi}(X,\omega)$ with
$d_{\chi}(\f, P_{\omega}(\min(\f,\p))) \leq d_{\chi}(\f,\p)$;
\item if $\p \leq \f$, 
 then $d_{\chi}(\f, \p)$ is comparable to  $\int_X \chi \circ (\f- \p) MA(\p)$.
\end{itemize}
\end{thm}

We refer   to \cite{Darv15} for convex weights with polynomial growth,
and to \cite[Proposition 5.3, Theorem 5.7, Theorem 6.1]{Darv21} 
for concave weights.

   \section{Quasi-monotone convergence} \label{sec:qm}

\subsection{Capacity vs quasi-monotonicity}
 
 Continuity of complex \MA operators along monotone sequences lies 
 at the heart of Bedford-Taylor theory \cite{BT82}. Several extensions
 of this continuity property have been proposed over the last decades
 under various restricted types of convergence.
 The following notion seems to encompass many of the latter:

  \begin{defi}  \label{def:quasimon}
  A sequence $\f_j \in PSH(X,\omega)$ converges quasi-monotonically 
   to $\f \in PSH(X,\omega)$ if there exists an increasing 
   (resp. decreasing) sequence $u_j $ (resp. $v_j$)
   in $PSH(X,\omega)$  such that 
   $u_j \leq \f_j \leq v_j$ for all $j$ and $u_j,v_j \rightarrow \f$ in $L^1(X)$.
  \end{defi}
  
  It follows easily from the definition that $\f_j$ converges to $\f$ in $L^1(X)$.
Observe that
$\p_j^+:=\sup_{\ell \geq j} \f_{\ell}$ decreases to $\f$, while
$\p_j^-:=\inf_{\ell \geq j} \f_{\ell}$ increases to $\f$.
However none of these functions usually belongs to $PSH(X,\omega)$:
$\p_j^+$ satisfies the mean value inequalities but it is no longer
u.s.c., while $\p_j^-$ is u.s.c. but does not satisfy the mean value inequalities.
It follows from \cite[Proposition 8.4]{GZbook} that
$$
\f_j^+:=\left( \sup_{\ell \geq j} \f_{\ell} \right)^* \in PSH(X,\omega)
$$  
and decreases to $\f$ pointwise.
By duality we consider the sequence
$$
\f_j^-:=P_{\omega} \left( \inf_{\ell \geq j} \f_{\ell} \right).
$$
The latter belongs to $PSH(X,\omega)$ as soon as it is not identically $-\infty$,
which is the case if $\f_j \geq u_j$ with $u_j$ increasing, since we then 
obtain $\f_j^- \geq u_j$. We thus obtain the following reformulation
of the quasi-monotone convergence:

\begin{lem} \label{lem:def}
A sequence $\f_j \in PSH(X,\omega)$ converges quasi-monotonically 
   to $\f \in PSH(X,\omega)$ if and only if
  $
\f_j^-:=P_{\omega} \left( \inf_{\ell \geq j} \f_{\ell} \right)
$
increases to $\f$. 
\end{lem}

It is a celebrated result of \cite{BT82} 
that  if $\f_j^-$ converges to $\f$ in $L^1$,
the convergence  
 holds pointwise off a pluripolar set.

    \begin{exa} \label{exa:quasilowerbound}
    Assume $\f_j \in PSH(X,\omega)$ converges in $L^1(X)$
 to $\f \in PSH(X,\omega)$, and assume there exists $\f \geq \p \in PSH(X,\omega)$
 and $\e_j \in \R^+$ decreasing to $0$ such that $\f_j \geq (1-\e_j) \f+\e_j\p$.
 Then $\f_j^-\geq (1-\e_j) \f+ \e_j \p$, hence $\f_j$ converges to $\f$ quasi-monotonically.
  In this case it has been observed in \cite[Lemma 1.2]{GLZ19} that 
 $\f_j$ converges to $\f$ in capacity.
This is a special case of  Theorem \ref{thm:qmvscap} below.
  \end{exa}

  It is well-known that monotone convergence implies convergence in capacity
  (see \cite[Proposition 4.25]{GZbook}).
  We extend this here to quasi-monotone convergence.
  
  \begin{thm} \label{thm:qmvscap}
  If a sequence $\f_j \in PSH(X,\omega)$ converges quasi-monotonically, then
  it converges in capacity.
  
  Conversely if $\f_j \in PSH(X,\omega)$ is uniformly bounded and converges in capacity,
  then a subsequence converges quasi-monotonically.
  \end{thm}
  
  The converse does not hold without a uniform lower bound on the $\f_j$'s, 
  as shown in Example \ref{exa:infty}
  which provides a sequence $(\f_j)$ which converges in capacity 
  while $\f_j^- \equiv -\infty$ (even after extracting).
  Example \ref{exa:extraction} moreover shows that it is usually necessary to extract,
  in order to reach the quasi-monotone convergence.

   \begin{proof}
   We let the reader check that 
   if $\f_j$ converges quasi-monotonically, then for all $C>0$,
   $\max(\f_j,-C)$ converges quasi-monotonically
   to $\max(\f,-C)$. Since the convergence in capacity 
   of $\f_j$ to $\f$ is equivalent to the 
   convergence in capacity of 
   $\max(\f_j,-C)$   to $\max(\f,-C)$,
   we are reduced to the uniformly bounded case.
   
   \smallskip
   
Assume first that $\f_j$ converges quasi-monotonically to $\f \in PSH(X,\omega)$.
Fix $u \in PSH(X,\omega)$ such that $-1 \leq u \leq 0$.
 It follows from  
 Proposition \ref{pro:blocki} that for all $\delta>0$,
 \begin{eqnarray*}
 \lefteqn{
  \int_{\{ |\f_j-\f| \geq \delta \}} (\omega+dd^c u)^n
 \leq  \delta^{-(n+1)} \int_X (\f_j^+-\f_j^-)^{n+1}(\omega+dd^c u)^n} \\
 & \leq & \delta^{-(n+1)} (n+1)! \sum_{j=0}^n \int_X (\f_j^+-\f_j^-)^{n+1-j}
 (\omega+dd^c \f_j^-)^j \wedge \omega^{n-j}.
 \end{eqnarray*}
 Each of the above integral converges to zero as $j \rightarrow +\infty$,
 as follows from Bedford-Taylor monotone convergence theorem. Thus 
 $\f_j$ converges in capacity.
 
 \smallskip
 
 Assume now that $\f_j$ converges in capacity. 
 Rescaling $\omega$, we can assume that $-1 \leq \f_j \leq 0$.
 Let $\delta_j,\e_j>0$ be sequences decreasing to zero.
 Observe that 
 $$
 \{ |\f_j-\f_{j+1}| \geq \delta \} \subset 
 \{ |\f_j-\f| \geq \delta/2 \} \cup \{ |\f_{j+1}-\f| \geq \delta/2 \}.
 $$
 Extracting and relabelling, we can thus assume that 
 $$
 {\rm Cap}_{\omega}\left( \{ |\f_j-\f_{j+1}| \geq \delta_j \} \right) \leq \e_j-\e_{j+1}.
 $$
 
 Set $E_j=\{ \f_{j+1} \leq \f_j-\delta_j \}$ and $F_j=\cup_{\ell \geq j} E_{\ell}$.
 The sequence $j \mapsto F_j$ is decreasing with $ {\rm Cap}_{\omega}(F_j) \leq \e_j$. 
 Fix $A_j \geq 1$ and consider
 $$
 h_j:=\left( \sup \left\{ u \in PSH(X,\omega), \, u \leq -A_j \text{ on } F_j
 \text{ and } u \leq 0 \text{ on } X \right\} \right)^*.
 $$
 This is a variant of the "relative extremal function" considered in \cite[Definition 9.13]{GZbook}.
 Adapting \cite[Section 9.3]{GZbook}, one easily obtains the following:
 \begin{itemize}
 \item $h_j \in PSH(X,\omega)$ with $-A_j \leq h_j \leq 0$;
 \item $h_j=-A_j$ on $F_j \setminus P_j$, where $P_j$ is a pluripolar set;
 \item ${\rm Cap}_{\omega}(F_j) \geq A_j^{-n-1} \int_X (-h_j) (\omega+dd^c h_j)^n$.
 \end{itemize}
 Thus $\int_X (-h_j) (\omega+dd^c h_j)^n \leq 1$ if we choose $A_j^{n+1} \e_j=1$ and
 it follows from Stokes theorem that
 $$
 \int_X (-h_j) \omega^n \leq \int_X (-h_j) (\omega+dd^c h_j)^n \leq 1.
 $$
We infer that  $\sum_{\ell \geq 1} 2^{-\ell} h_{\ell} \in PSH(X,\omega)$, hence the sequence
$$
H_j:=\sum_{\ell \geq j} 2^{-\ell-1} h_{\ell} \in PSH(X,2^{-j}\omega)
$$
increases to zero as $j \rightarrow +\infty$.

We set $\p_j:=(1-2^{-j}) \f_j+H_j-2^{-j+1} \in PSH(X,\omega)$.
Since $H_j \leq 0$ and $\f_j \geq -1$, we obtain $\p_j \leq \f_j$.
We choose $\delta_j=2^{-j-1}/(1-2^{-j})$ and $A_j=2^{j+1}$
 and claim that   $j \mapsto \p_j$ increases to $\f$ as $j \rightarrow +\infty$.
Indeed
\begin{itemize}
\item either $x \in X \setminus F_j$, then $\f_{j+1}(x) \geq \f_j(x)-\delta_j$ hence
$$
\p_{j+1}-\p_j=(1-2^{-j})(\f_{j+1}-\f_j)+2^{-j-1}\f_{j+1}-2^{-j-1}h_j +2^{-j} \geq 0
$$
using that $-h_j \geq 0$ and $\f_{j+1} \geq -1$;
\item or $x \in F_j$ and we obtain
$\p_{j+1}-\p_j \geq -1-2^{-j-1}h_j=0$ if $x \notin P_j$.
\end{itemize}
Thus $\p_{j+1} \geq \p_j$ a.e. hence everywhere.
Replacing $\p_j$ by $\max(\p_j,-1)$, we moreover
obtain an increasing sequence which is uniformly bounded.
  \end{proof}

    \subsection{Finite energy sequences}
    
  Following Theorem \ref{thm:qmvscap} we now prove our main Theorem,
  which extends several partial  results previously obtained 
  (\cite[Theorem 2]{Xing96}, \cite[Theorem 2.1]{H08},
    \cite[Theorem 1.2]{Xing09},\cite[Theorem 5.7]{BBGZ13},
    \cite[Corollary 5.7]{Darv15}, \cite[Propositions 2.6 and 6.4]{BDL17}, \cite[Proposition 1.9]{GLZ19},
    \cite[Proposition 5.7]{Tr22}, \cite[Theorem 1.2]{Gup22}).

     \begin{thm} \label{thm:energy}
    Assume $\f_j \in PSH(X,\omega)$  converges in $L^1(X)$
 and fix   $\chi \in {\mathcal W}$ a weight which is either convex
 or concave with polynomial growth at $-\infty$.
   \begin{itemize}   
   \item If $\f_j$ converges 
   in capacity and $\f_j \geq \p$ for  
   some $\p \in {\mathcal E}_{\chi}(X,\omega)$, 
   then $\f_j$ converges  in $({\mathcal E}_{\chi}(X,\omega),d_{\chi})$.
   \item If $\f_j$ converges  in $({\mathcal E}_{\chi}(X,\omega),d_{\chi})$, then 
   $\f_j$ converges in capacity. Moreover
   up to extracting and relabelling,
   $\f_j$ converges quasi-monotonically and there exists 
   $\p \in {\mathcal E}_{\chi}(X,\omega)$ such that    $\f_j \geq \p$.   
   \end{itemize}
   \end{thm}

       For sequences that are not uniformly bounded, one cannot expect that 
    quasi-monotone convergence is equivalent to convergence in energy.
    For instance if $\p \in PSH(X,\omega)$ has some positive Lelong number and 
    $\e_j$ decreases to zero, then 
       $\f_j=\e_j \p$ converges quasi-monotonically to $\f=0$, but
         not in energy. 
 A finite energy lower bound turns out to be a necessary and sufficient condition.

    \begin{proof}
   We start with the first item.
   Assume first that the sequence $(\f_j)$ is uniformly bounded.
   When $\chi(t)=t$, it follows from \cite[Proposition 2.3]{BBEGZ} that 
    $\f_j \rightarrow \f$ in $({\mathcal E}^1(X,\omega),d_1)$ if and only if
    $\int_X (\f_j-\f) [MA(\f)-MA(\f_j)] \rightarrow 0$.
    The latter holds when $\f_j$ converges in capacity, as follows from
    Proposition \ref{pro:cvcap}.
 More generally the convergence in $({\mathcal E}_{\chi}(X,\omega),d_{\chi})$
 is equivalent to the following
 $$
 \int_X \left| \chi (\f_j-\f) \right| \left[ MA(\f_j)+MA(\f) \right] \rightarrow 0,
 $$
 which is a consequence of \cite[Theorem 4.26]{GZbook}.
   
    We now reduce the general case to the uniformly bounded one.
   Fix  $ \tilde{\chi}$ a weight such that 
    $\int_X (-\tilde{\chi} \circ \p) \, MA(\p) <+\infty$ and 
     $\tilde{\chi}(t)/\chi(t) \rightarrow +\infty$ as $t \rightarrow -\infty$
     (see \cite[Exercise 10.5]{GZbook}).
    Set $\f_j^C=\max(\f_j,-C)$ and $\f^C=\max(\f,-C)$,
 it follows from  Theorem \ref{thm:darvEchi}   that
  \begin{eqnarray*}
  d_{\chi}(\f_j,\f_j^C) &\sim &\int_X \chi \circ (\f_j^C -\f_j) \, MA(\f_j) 
   \leq   \int_{\{ \f_j <-C\}} (-\chi \circ \f_j) \, MA(\f_j) \\
   &\leq & \frac{\chi(-C)}{\tilde{\chi}(-C)} \int_X (-\tilde{\chi} \circ \f_j) \, MA(\f_j) 
   \leq  \frac{M^n \chi(-C)}{\tilde{\chi}(-C)} \int_X (-\tilde{\chi} \circ \p) \, MA(\p) ,
  \end{eqnarray*}
  where the last inequality follows from \cite[Lemmas 2.3 and 3.5]{GZ07}.
       The conclusion thus follows from the
   uniformly bounded case and the triangle inequality.
  
  \smallskip
  
  It thus remains to show the second item.
  To check that $\f_j$ converge in capacity, it suffices to show
  that any subsequence admits a subsubsequence that converges in capacity.
  Extracting and relabelling, we can assume that 
  $d_{\chi}(\f_j,\f_{j+1}) \leq 2^{-j}$.
  Set $\f_{j,k}^-:=P_{\omega}(\min_{j \leq \ell \leq j+k} \f_{\ell})$.
  A repeated use of  Theorem \ref{thm:darvEchi} ensures that 
  $\f_{j,k}^- \in {\mathcal E}_{\chi}(X,\omega)$ with
  \begin{eqnarray*}
   d_{\chi}(\f_j,\f_{j,k}^-) &\leq &
  d_{\chi} \left(\f_j, P_{\omega} \left(\min_{j+1 \leq \ell \leq j+k} \f_{\ell} \right) \right) \\
  &\leq & d_{\chi}(\f_j,\f_{j+1})+
 d_{\chi} \left(\f_{j+1}, P_{\omega} \left(\min_{j+1 \leq \ell \leq j+k} \f_{\ell} \right) \right) \\
 &\leq & d_{\chi}(\f_j,\f_{j+1})+
 d_{\chi} \left(\f_{j+1}, P_{\omega} \left(\min_{j+2 \leq \ell \leq j+k} \f_{\ell} \right) \right)  \\
 &\leq & \sum_{\ell=j}^{j+k-1} d_{\chi}(\f_{\ell},\f_{\ell+1})
 \leq 2^{-j+1}.
   \end{eqnarray*}
     We infer that $k \mapsto \f_{j,k}^-$ decreases, as $k$ increases to $+\infty$,
   to $\f_j^- \in {\mathcal E}_{\chi}(X,\omega)$ with 
   $ d_{\chi}(\f_j,\f_{j}^-) \leq 2^{-j+1}$. 
   It follows that $\f_{j}$ converges to $\f$
   quasi-monotonically and for all $j$, $\f_j \geq \p:=\f_1^- \in {\mathcal E}_{\chi}(X,\omega)$.
   In particular $\f_j$ converges to $\f$ in capacity.
    \end{proof}

    \section{Examples and remarks} \label{sec:exa}
    
    \subsection{Quasi-monotone convergence of Monge-Amp\`ere potentials}
    
    Families of solutions to complex Monge-Amp\`ere equations or flows provide
    natural examples of sequences which converge quasi-monotonically.
    We illustrate this here with two typical situations.
    
    \begin{exa} \label{exa:solutions}
    Let $\mu$ be a non pluripolar probability measure.
    It follows from \cite{GZ07,Din09} that there exists a unique
    function $\f \in {\mathcal E}(X,\omega)$ such that
    $$
    MA(\f)=e^{\f} \mu.
    $$
    \cite[Theorem C]{GLZ19} and Choquet's lemma
    show that $\f$ is the quasi-monotone limit of a sequence of functions
    $\f_j \in {\mathcal E}(X,\omega)$ such that
    $MA(\f_j) \leq e^{\f_j} \mu$.
    \end{exa}

    We now consider smoothing properties of the K\"ahler-Ricci flow 
    \cite{GZ17,DL17}.
    
      \begin{exa} \label{exa:fkr}
      Assume that $X$ is a Calabi-Yau manifold and fix
    $T_0=\omega+dd^c \f_0$  a positive closed current 
  with zero Lelong numbers which is cohomologous to $\omega$.
      It has been shown in \cite{GZ17,DL17} that
      there exists a unique family of K\"ahler forms
      $(\omega_t)_{t>0}$ on $X$, which evolve along the K\"ahler-Ricci flow
  $$
  \frac{\partial \omega_t}{\partial t}=-{\rm Ric}(\omega_t),
  $$
 and such that $\omega_t \rightarrow T_0$ as $t \rightarrow 0$. 
 The (normalized) potentials $\f_t \in PSH(X,\omega)$ of $\omega_t=\omega+dd^c \f_t$
 are solutions of a complex \MA flow,
 $$
 (\omega+dd^c \f_t)^n=e^{\partial_t \f_t+h} \omega^n,
 $$
 and the convergence $\f_t \rightarrow \f_0$ at time zero 
 is such that $\f_t \geq \f_0 -A(t-t\log t)$ for some constant $A \geq 0$
 \cite[Lemma 2.9]{GZ17}.
 It follows that
 $$
 \f_t^-:=P_{\omega} \left( \inf_{0 < s \leq t} \f_s \right) \geq \f_0 -A(t-t\log t),
 $$
 hence the convergence of $\f_t$ towards $\f_0$ is quasi-monotone.
    \end{exa}

     \subsection{Intermediate convergences}

  We first provide  examples   of $\omega$-psh functions
  $\f_j$ such that $\f_j^-$ is identically $-\infty$, while 
  $\f_j$ converges in capacity to some $\omega$-psh function $\f$.

\begin{exa} \label{exa:infty}
Consider the Riemann sphere endowed with the Fubini-Study K\"ahler form $(X,\omega)=(\PP^1,\omega_{FS})$. Then
$$
\f_j[z]=\log|z_1-\tau_j z_0|-\log |z| \longrightarrow \f[z]=\log|z_1|-\log |z|
$$
if the $\tau_j$'s converge to $0$,
and $P_{\omega}(\min(\f_j,\f_{j+1})) \equiv -\infty$ because 
functions in $PSH(\PP^1,\omega_{FS})$ can have at most one Lelong number of size $1$.
On the other hand the convergence of $\f_j$ towards $\f$ is uniform on compact subsets of
$\PP^1 \setminus [1:0]$, so the sequence converges in $L^1$ and in capacity.

More generally if $\p \in PSH(\PP^n,\omega_{FS})$ the set 
$E_1(\p)=\{ x \in \PP^n , \; \nu(\p,x) \geq 1 \}$ has to be included in
a hyperplane \cite[Proposition 2.2]{CG09}, so one can cook up similar
 examples for which $\f_j^-$ is identically $-\infty$.
\end{exa}

We now provide an example of a uniformly bounded sequence $(\f_j)$
of $\omega$-psh functions which converge to $0$ in capacity with $\f_j^- \equiv -1$.
This shows that it is necessary to use extractions in Theorem \ref{thm:energy}.

\begin{exa} \label{exa:extraction}
Using local charts, we construct for each point $a \in X$ a function
$G_a \in PSH(X,\omega)$ which has a logarithmic singularity at point $a$.

Observe that $\max(G_a,-1) \equiv -1$ in a neighborhood $V_a$ of $a$.
We cover $X$ by finitely many such neighborhoods $V^1_{a_{1,1}},\ldots V^1_{a_{s_1,1}}$.
We similarly cover $X$ by finitely many neighborhoods $V^2_{a_{1,2}},\ldots,V^2_{a_{s_2,2}}$
on which the function $\max(G_a/2,-1)$ is identically $-1$.
The latter neighborhoods are smaller and we need more points.

We go on by induction, dividing at each step $G_a$ by an extra factor $2$, considering the family
of functions $\max(G_a/2^n,-1)$ at step $n$. We then label the corresponding sequence of functions
$(\f_j)$, so that 
$\f_j=\max(G_a/2^n,-1) \in PSH(X,\omega)$ for some $a_j=a_{l,n}$ with $1 \leq \ell \leq s_n$
and $n=n_j \rightarrow +\infty$ as $j \rightarrow +\infty$.
By construction we obtain $\f_j \geq -1$ and $\f_j^- \equiv -1$, since
$\f_j^-$ lies below each function from a fixed step.
On the other hand
$G_{a_j}/2^{n_j} \rightarrow 0$ in capacity, hence $\f_j \rightarrow 0$ in capacity.
\end{exa}

It has been observed by Cegrell in \cite{Ceg83}  that the \MA operator is not 
 continuous for the $L^1_{loc}$-topology.
  The following is an explicit example of this phenomenon, adapted to the compact context.
   
   \begin{exa} \label{exa:pascap}
   Assume $(X,\omega)=(\PP^2,\omega_{FS})$ and
consider 
$$
\f_j[z]=\frac{1}{j}\max \left( \log \left| z_0^j+z_1^j+z_2^j \right| , \log |z_0| \right)-\log |z|.
$$
Observe that $\f_j \in PSH(X,\omega)$ is locally bounded outside the finite set
$$
F_j=\{ [z] \in \PP^2, \, z_0=0
\; \& \; z_{1}^j+z_2^j=0 \} 
$$ 
which is included in the circle $S^1=\{z_0=0 \; \& \; |z_1|=|z_2| \}$.
Note also that $\f_j$ converges in $L^1$ to
$
\f[z]=\max_{0 \leq j \leq 2} \log^+|z_j|-\log|z|.
$
The \MA measures $MA(\f_j)$ are 
combination of Dirac masses at points of $F_j$ and
converge to the Haar measure on $S^1$, while $MA(\f)$ is the Haar measure on the torus
$$
{\mathbb T}^2=\{ [z] \in \PP^2, \; |z_0|=|z_1|=|z_2| \}.
$$
Thus $\f_j$ does not converge in capacity to $\f$.
\end{exa}

We finally compare the various types of convergence in energy classes.

  \begin{exa} \label{exa:energy}
Assume $(X,\omega)=(\PP^1,\omega_{FS})$ and consider
$$
\f_j[z]=\e_j \max (\log|z_1|-\log|z|,-C_j) \in PSH(X,\omega),
$$
 where $0 \leq \e_j \leq 1$ and $C_j \geq 0$. These examples are toric, so one can use the dictionary established in \cite{CGSZ19} to justify the following assertions:
 \begin{itemize}
 \item $\f_j \rightarrow 0$ in $L^1$ iff it does so in capacity/quasi-monotonically
  iff $\e_j \rightarrow 0$;
 \item $\f_j \rightarrow 0$ in  $({\mathcal E}_{\chi}(X,\omega),d_{\chi})$ 
 iff $\e_j \chi(-\e_jC_j) \rightarrow 0$.
 \end{itemize}
 By considering weights  with arbitrarily slow growth, we conclude that 
 $\f_j \rightarrow 0$  in some $({\mathcal E}_{\chi}(X,\omega),d_{\chi})$ 
 as soon as $\e_j \rightarrow 0$,
 whatever the speed at which $C_j \nearrow +\infty$.
\end{exa}

  \subsection{Concluding remarks}
  

  \subsubsection{Independence on $\omega$}
  
 Let $\tilde{\omega}=\omega+dd^c \rho$ be a K\"ahler form cohomologous to $\omega$.
 Then $\f_j \in PSH(X,\omega)$ if and only if $\p_j=\f_j-\rho \in PSH(X,\tilde{\omega})$.
 We let the reader check that $\f_j$ converges quasi-monotonically if and only if 
 so does $\p_j$ (and similarly for the other notions of strong convergence).

Assume now $\tilde{\omega}$ is an arbitrary K\"ahler form. We want to compare notions of strong convergence with respect to $\omega$ and with respect to $\tilde{\omega}$. We claim that
these are essentially the same. Using $\omega+\tilde{\omega}$ as a third auxiliary form, we see that
it suffices to treat the case when $\omega \leq \tilde{\omega}$.
The following are left to the reader:
\begin{itemize}
\item $PSH(X,\omega) \subset PSH(X ,\tilde{\omega})$ and
if $\f \in {\mathcal E}_{\chi}(X,\omega)$ then $\f \in {\mathcal E}_{\chi}(X,\tilde{\omega})$;
\item if  $\f \in {\mathcal E}_{\chi}(X,\tilde{\omega})$ 
then $P_{\omega}(\f) \in {\mathcal E}_{\chi}(X,\omega)$;
\item $\f_j \in PSH(X,\omega)$ converges   to $\f \in PSH(X,\omega)$ with respect to $Cap_{\omega}$
if and only if it does so with respect to $Cap_{ \tilde{\omega}}$;
\item if $\f_j \in PSH(X,\omega)$ converges quasi-monotonically to $\f \in PSH(X,\omega)$,
then the same property holds with respect to  $\tilde{\omega}$.
\end{itemize}

Adapting Example \ref{exa:infty}, one can find
$(\f_j) \in PSH(X,\tilde{\omega})^{\N}$ 
which converges quasi-monotonically to $\f \in PSH(X,\tilde{\omega})$
with $P_{\omega}(\inf_{\ell \geq j} \f_{\ell}) \equiv -\infty$. This converse however holds
if we assume an appropriate lower bound on the sequence.

In particular a uniformly bounded sequence  $\f_j \in PSH(X,\omega) \cap PSH(X,\tilde{\omega})$
 converges quasi-monotonically w.r.t. $\omega$ if and only if it does so 
 w.r.t. $\tilde{\omega}$.

  \subsubsection{Big classes and prescribed singularities}
  
  All notions introduced previously and all properties established so far
  can be adapated to the case when the reference form $\omega$ is no longer
  K\"ahler, but merely a smooth closed real $(1,1)$-form representing a {\it big cohomology class}.
  We refer the reader to \cite{BEGZ10} for basics of pluripotential theory in that context.
    One can also extend these results to the case of big classes with prescribed singularities,
  a theory that has been developed by Darvas-DiNezza-Lu in \cite{DDL18,DDL19}
  and further studied in \cite{Tr20,Tr22}.

   \subsubsection{The local setting}
   
   Let $\Omega$ be a pseudoconvex domain of $\C^n$.
   We let $PSH(\Omega)$ denote the set of \psh functions in $\Omega$.
   
   \begin{defi}
 A sequence $(\f_j)$   of \psh functions in $\Omega$ 
  converges to $\f \in PSH(\Omega)$ quasi-monotonically if
 $
\f_j^-:=P_{\Omega} \left( \inf_{\ell \geq j} \f_{\ell} \right) 
 $
 increases to $\f$.
   \end{defi}
   
   Here $P_{\Omega}(h)$ denotes the largest \psh function in $\Omega$ lying below $h$.
   Adapting what we have done in the compact case, we can  establish that
   \begin{itemize}
   \item quasi-monotone convergence implies convergence in capacity;
   \item a sequence can converge in capacity but not quasi-monotonically;
   \item both notions essentially  coincide for uniformly bounded sequences.
   \end{itemize}
We leave the details to the reader.

  \end{document}